\documentclass[reqno,12pt]{amsart} 
\usepackage{mathrsfs}
\usepackage[margin=1in]{geometry}
\usepackage{amssymb,amsmath,amsfonts}
\usepackage{latexsym}
\usepackage{amssymb}
\usepackage{amsmath}
\usepackage{amsfonts}
\usepackage{amscd}
\usepackage{texdraw}
\usepackage{color}
\usepackage{cite}
\usepackage[centertags]{amsmath}
\usepackage[colorlinks=true,urlcolor=blue,linkcolor=red,citecolor=red]{hyperref}
\usepackage[pagewise]{lineno}

\theoremstyle{plain}
\newtheorem{definition}{Definition}[section]
\newtheorem{theorem}[definition]{Theorem}
\newtheorem{lemma}[definition]{Lemma}
\newtheorem{proposition}[definition]{Proposition}
\newtheorem{corollary}[definition]{Corollary}
\newtheorem{remark}[definition]{Remark}
\newtheorem{example}[definition]{Example}

\allowdisplaybreaks
\begin{document}
\title[Inverse-closedness of Schur and BGS type algebras]{Inverse-closedness of weighted Schur and BGS type quasi-Banach algebras}

\author[P. A. Dabhi]{Prakash A. Dabhi}
\address{Department of Basic Sciences, Institute of Infrastructure Technology Research and Management (IITRAM), Maninagar (East), Ahmedabad -380026, Gujarat, India}
\email{lightatinfinite@gmail.com, prakashdabhi@iitram.ac.in}

\author[K. B. Solanki]{Karishman B. Solanki}
\address{Department of Mathematics, Indian Institute of Technology Ropar, Rupnagar, Punjab -140001, India}
\email{karishsolanki002@gmail.com, staff.karishman.solanki@iitrpr.ac.in}



\date{}

\dedicatory{}

\commby{}

\begin{abstract}
We prove that the weighted quasi-Banach algebras of operator valued matrices satisfying Schur and Baskakov-Gohberg-Sj\"ostrand (BGS) conditions are inverse-closed in the Banach algebra $B(\ell^2(X,\mathcal{H}))$ whenever the weight is admissible, where $\mathcal{H}$ is a Hilbert space and $X$ is a relatively separated subset of $\mathbb{R}^d$. Furthermore, we identify the Gel'fand space of weighted infinite variable group algebra $\ell^p_\omega(\mathbb{Z^N})$ for $0<p\leq1$, and establish inverse-closedness of infinite variable analogue of BGS-type algebra in $B(\ell^2(\mathbb{Z^N},\mathcal{H}))$.
\end{abstract}

\maketitle

\section{Introduction}
Let $\mathcal{A}$ be a subalgebra of an algebra $\mathcal{B}$ both having same unit. Then $\mathcal{A}$ is \textit{inverse-closed} in $\mathcal{B}$ if $a\in\mathcal{A}$ is invertible in $\mathcal{B}$ implies $a$ is invertible in $\mathcal{A}$. The notion of inverse-closedness traces back to the classical result of Wiener \cite{wi} which states that if $C(\mathbb{T})$ is the collection of all complex valued continuous functions on the unit circle $\mathbb{T}$, $A(\mathbb{T})$ is the collection of all such functions having absolutely convergent Fourier series, and if $f\in A(\mathbb{T})$ is nowhere vanishing, then $\frac{1}{f}\in A(\mathbb{T})$. This result is well known as Wiener's theorem or Wiener's lemma in the literature and can be stated as: ``$A(\mathbb{T})$ is inverse-closed in $C(\mathbb{T})$". Its weighted analogues were first studied by Domar in \cite{do} for non-quasianalytic weights and by Bhatt and Dedania in \cite{bh} for more general weights. For $0<p<1$, $p$-power analogue of Wiener's theorem is obtained by \.Zelazko in \cite{ze} and its weighted analogues are obtained in \cite{De}. In \cite{Da}, the $p$-power analogues for $1<p<\infty$ are studied by taking some conditions on weight. In particular, it follows that if $0<p<\infty$, $\omega$ is an admissible weight on $\mathbb{Z}$, which is a $p$-algebra weight when $p>1$, and if $A^p_\omega(\mathbb{T})$ is the algebra of all $f\in C(\mathbb{T})$ satisfying $\sum_{n\in\mathbb{Z}} |\widehat{f}(n)|^p\omega(n)^p<\infty$, then $A^p_\omega(\mathbb{T})$ is inverse-closed in $C(\mathbb{T})$. If $\mathcal{A}$ is inverse-closed in $\mathcal{B}$, then $\mathcal{A}\subset\mathcal{B}$ is called a \textit{Wiener pair} as it goes back to the work of Wiener. Bochner and Phillips were the first to investigate the vector-valued analogues of Wiener's theorem in \cite{bo}, and more general weighted results are obtained in \cite{kb,kbmul}. We refer the reader to \cite{BaCaHe,BaCaHe2,GroLei,kbwp,kbmul,Jaf,Ku1,KB,Sjo1,Sjo,Su2,Su1} and references therein for other generalization of Wiener's theorem and its various applications in many fields of analysis; of course the list is not exhaustive. 

A $p$-Banach $\ast$-algebra $\mathcal{A}$ is \emph{symmetric} if $\sigma(aa^\ast)\subset [0,\infty)$ for all $a\in\mathcal{A}$ or equivalently $\sigma(a)\subset\mathbb{R}$ for all $a=a^\ast\in\mathcal{A}$. Symmetric algebras plays a crucial role in the theory of Banach algebras because they share several properties of $C^\ast$-algebras. Observe that symmetry gives information about a given algebra while inverse-closedness gives information about two nested algebras but the interesting fact is that these two topics are closely related to such an extent that most of the time the symmetry of a $p$-Banach $\ast$-algebra $\mathcal{A}$ is shown using inverse-closedness of $\mathcal{A}$ in some $C^\ast$-algebra and one of the important technical result to do so is the  Hulanicki's lemma (see Theorem \ref{thm:Hul}). For more details on relation between Hulanicki's lemma and symmetric Banach $\ast$-algebras refer to \cite{Ba,Fol,Fe,GroLei,Rim,Sh,Shi} and references therein, again the list is not exhaustive. Hulanicki's lemma for quasi algebras has been obtained in \cite{kbhul} along with an application deriving the Barnes' lemma and Wiener's theorem for twisted convolution.

Motivated by the above works and particularly the works in \cite{KB, Ba, Ba1, Kri}, we discuss the inverse-closedness of weighted Schur and Baskakov-Gohberg-Sj\"ostrand (BGS) type quasi-normed algebras in $B(\ell^2(X,\mathcal{H}))$, where $X$ is a relatively separated subset of $\mathbb{R}^d$ and $\mathcal{H}$ is a Hilbert space. 

Let $0< p<\infty$, $\omega$ be a weight on $\mathbb{R}^d$, and let $X\subset \mathbb R^d$ be a \textit{relatively separated} set, that is,
$$\sup_{n\in \mathbb{Z}^d}|X\cap (n+[0,1]^d)|<\infty.$$
Note that relatively separated sets are countable. Let $\mathcal{H}$ be a Hilbert space, and let $B(\mathcal{H})$ be the $C^\ast$-algebra of all bounded linear operators on $\mathcal{H}$ equipped with the operator norm. 

Define the \textit{weighted Schur-type algebra} over $X$, denoted by $\mathcal{S}^p_\omega=\mathcal{S}_\omega^p(X)$, as the collection of all matrices $A=[A_{k,l}]_{k,l\in X}$, with each entry $A_{k,l}\in B(\mathcal H)$, that satisfies the following condition:
\begin{enumerate}
    \item If $0<p\leq1$, then 
    \begin{align} \label{def:S^p_nu norm p<1}
        \|A\|_{\mathcal S_\omega^p} = \max\left\{\sup\limits_{k\in X}\sum\limits_{l\in X}\|A_{k,l}\|^p\omega(k-l)^p,\sup\limits_{l\in X}\sum\limits_{k\in X}\|A_{k,l}\|^p\omega(k-l)^p\right\} <\infty.
    \end{align}

    \item If $1<p<\infty$, then
    \begin{align} \label{def:S^p_nu norm p>1}
        \|A\|_{\mathcal S_\omega^p} = \max\left\{\sup\limits_{k\in X}\left(\sum\limits_{l\in X}\|A_{k,l}\|^p\omega(k-l)^p\right)^{\frac{1}{p}},\sup\limits_{l\in X}\left(\sum\limits_{k\in X}\|A_{k,l}\|^p\omega(k-l)^p\right)^{\frac{1}{p}}\right\}<\infty.
    \end{align}
\end{enumerate}
And the \textit{weighted BGS-type algebra}, denoted by $\mathcal{C}^p_\omega=\mathcal{C}^p_\omega(\mathbb{Z}^d)$, is defined as collection of all $B(\mathcal{H})$-valued matrices $A=[A_{k,l}]_{k,l\in\mathbb{Z}^d}$ satisfying
\begin{align} \label{eq:Cvnorm}
    |A| = \sum_{l\in\mathbb{Z}^d} \left(\sup_{k\in\mathbb{Z}^d} \|A_{k,k-l}\|\right)^p \omega(l)^p<\infty.
\end{align}

A weight on $\mathbb{R}^d$ is said to satisfy the \textit{weak growth condition} if
\begin{align} \label{def:weakgrowthweight}
\omega(x)\geq C (1+|x|)^\delta \quad \text{for some constants} \quad C>0 \quad \text{and} \quad 0<\delta\leq1.    
\end{align}
In \cite{KB}, K\"ohldorfer and Balazs have shown that $\mathcal{S}^1_\omega(X)$ and $\mathcal{C}^1_\omega(\mathbb{Z}^d)$ are symmetric, provided the weight $\omega$ is admissible that further satisfies the weak growth condition for Schur-type algebra. Here we extend these results in the quasi-norm setting, that is, the case of $p<1$. Moreover, the case of $p>1$ is also treated for BGS-type algebra with additional conditions on the weight.

The following are the main theorems.

\begin{theorem} \label{thm:main Schur}
Let $0<p\leq1$, $\omega$ be an admissible weight satisfying the weak growth condition, and let $X\subset \mathbb R^d$ be a relatively separated set. Then the algebra $\mathcal{S}^p_\omega(X)$ is inverse-closed in $B(\ell^2(X,\mathcal{H}))$. In particular, $\mathcal{S}^p_{\omega}$ is symmetric.
\end{theorem}

\begin{theorem} \label{thm:main Cv}
Let $0<p<\infty$, and let $\omega$ be an admissible weight on $\mathbb{Z}^d$ such that for $p>1$, $\omega$ satisfies $\omega^{-p'}\star\omega^{-p'}\leq\omega^{-p'}$ and $\sum_{n\in\mathbb{Z}^d} \omega(n)^{-p'}<\infty$. Then the algebra $\mathcal{C}^p_\omega$ is inverse-closed in $B(\ell^2(\mathbb{Z}^d,\mathcal{H}))$. In particular, $\mathcal{C}^p_\omega$ is symmetric.
\end{theorem}
Here, $p'$ denotes the conjugate index of $p>1$, that is, $\frac{1}{p}+\frac{1}{p'}=1$, and $\star$ denotes the convolution. We shall use Hulanicki's lemma for $p$-Banach algebra (Theorem \ref{thm:Hul}) and multivariate $p$-power weighted analogue of Wiener's theorem (Theorem \ref{thm:Wiener}) for Schur-type and BGS-type algebras, respectively. 

Furthermore, for $0<p\leq1$ and a weight $\omega$ on $\mathbb{Z^N}$ (the set of finitely supported sequences of integers), we identify the Gel'fand space of $\ell^p_\omega(\mathbb{Z^N},\mathbb{C})$ with a subset of $\mathbb{C}^\infty$, the countable product of the complex plane. It is claimed in \cite{AT} that it can be identified with a countable product of closed annuli; however, this is not correct in general, and we provide an example showing that this result fails in certain cases. Moreover, we extend the inverse-closedness of BGS-type quasi algebras to the infinite variable setting. The identification of Gel'fand space is of independent interest, and one of the reasons is that it provides a simple proof of Wiener's theorem. In fact, the Wiener theorem used here and its different analogues in the literature essentially depend on it. This further clarifies the connection between the present results and the structure of Gel'fand space. 

The paper is organized as follows. In Section \ref{sec:pre}, required notations, definitions, and auxiliary results used in paper are provided. Sections \ref{sec:Schur} and \ref{sec:BGS} are devoted to the inverse-closedness of Schur-type and BGS-type algebra, respectively. In Section \ref{sec:Gelfand space}, we present the result for the Gel'fand space of $\ell^p_\omega(\mathbb{Z^N},\mathbb{C})$ and show inverse-closedness of infinite variable BGS-type algebra.

\section{Preliminaries} \label{sec:pre}
Throughout the paper, $\mathbb{N}$ is the set of positive integers, $\mathbb{Z}$ is the set of integers, $\mathbb{R}$ is the set of real numbers, $\mathbb{C}$ is the set of complex numbers, $\mathbb{T}=\{z\in\mathbb{C}:|z|=1\}$ is the unit circle in $\mathbb{C}$, and for $d\in\mathbb{N}$, $\mathbb X^d$ is $d$-copies of the space $\mathbb X$, where $\mathbb X$ is one of the above set. Moreover, $\mathcal{H}$ denotes a Hilbert space and $B(\mathcal{H})$ denotes the $C^\ast$-algebra of bounded linear operators on $\mathcal{H}$ with the operator norm. 

We state an important inequality that will be used in this paper.
\begin{align} \label{eq:ineqp<1}
    \text{If } 0<p\leq1 \text{ and } a,b\geq0, \text{ then } (a+b)^p\leq a^p +b^p.
\end{align}

For $0<p\leq1$ and an algebra $\mathcal{A}$, a mapping $\|\cdot\| : \mathcal{A} \to [0,\infty)$ is a \emph{$p$-norm} \cite{ze} on $\mathcal{A}$ if the following conditions hold for all $x,y\in \mathcal{A}$ and $\alpha \in \mathbb{C}$.
\begin{enumerate}
\item $\|x\|=0$ if and only if $x=0$;
\item $\|x+y\|\leq\|x\|+\|y\|$;
\item $\|\alpha x\|= |\alpha|^p \|x\|$;
\item $\|xy\|\leq\|x\|\|y\|$.
\end{enumerate}
The algebra $\mathcal{A}$ along with a $p$-norm $\|\cdot\|$ is a \emph{$p$-normed algebra}. If $\mathcal{A}$ is complete in the $p$-norm, then $(\mathcal{A},\|\cdot\|)$ is a \emph{$p$-Banach algebra}.
   
\begin{remark}
When $p=1$, the map $\|\cdot\|$ is a \emph{norm} on $\mathcal{A}$, and $(\mathcal{A},\|\cdot\|)$ is a \emph{normed algebra} or a \emph{Banach algebra} if it is complete. When $0<p<1$, the $p$-norm, $p$-normed algebra and $p$-Banach algebra are also referred to as \emph{quasi-norm}, \emph{quasi-normed algebra} or \emph{quasi-Banach algebra}, respectively. But we will stick to the $p$-nomenclature in the paper as it also clarifies the norm in study. 
\end{remark}

A $p$-normed ($p$-Banach) $\ast$-algebra is a $p$-normed ($p$-Banach) algebra along with an isometric involution $\ast$. A \emph{$p$-$C^\ast$-algebra} is a $C^\ast$-algebra $(\mathcal A,\|\cdot\|)$ with the \emph{$p$-$C^\ast$-norm} $|x|=\|x\|^p\;(x \in \mathcal A)$. Let $\mathcal{A}$ be a $p$-Banach algebra with unit $e$, and let $x\in\mathcal{A}$. The set $\sigma_\mathcal{A}(x)=\{\lambda\in\mathbb{C}: \lambda e-x \ \text{is not invertible in}\ \mathcal{A} \}$ is the \emph{spectrum} of $x$ in $\mathcal{A}$ and the number $r_\mathcal{A}(x)=\sup\{|\lambda|^p:\lambda\in\sigma_\mathcal{A}(x)\}$ is the \emph{spectral radius} of $x$. The spectral radius formula gives $r_\mathcal{A}(x)=\lim_{n\to\infty}\|x^n\|^\frac{1}{n}$. We shall write just $\sigma(x)$ and $r(x)$ when the algebra in consideration is clear. Refer to \cite{bd,ze} for more details.

If $\mathcal{A}$ is a subalgebra of a $p$-Banach $\ast$-algebra $\mathcal{B}$ and both have same unit, then for $x\in\mathcal{A}$, it follows that $\sigma_\mathcal{B}(x) \subset \sigma_\mathcal{A}(x)$ and $r_\mathcal{B}(x) \leq r_\mathcal{A}(x)$. Note that $\mathcal{A}$ is inverse-closed in $\mathcal{B}$ if and only if $\sigma_\mathcal{B}(x) = \sigma_\mathcal{A}(x)$ and thus $r_\mathcal{B}(x) = r_\mathcal{A}(x)$ for all $x\in\mathcal{A}$ (cf. Theorem \ref{thm:Hul}). Due to this relation, inverse-closedness is also referred to as \textit{spectral-invariance}.

A \emph{weight} on $\mathbb{R}^d$ is a Borel measurable map $\omega:\mathbb{R}^d \to [0,\infty)$, and it is
\begin{enumerate}
    \item \textit{submultiplicative} if $\omega(x+y)\leq \omega(x)\omega(y)$ for all $x,y \in \mathbb{R}^d$.
    \item \textit{symmetric} if $\omega(x)=\omega(-x)$ for all $x\in\mathbb{R}^d$.
    \item \textit{satisfies the GRS-condition} if $\displaystyle \lim_{n\to\infty}\omega(nx)^\frac{1}{n}=1$ for all $x\in\mathbb{R}^d.$
    \item \textit{admissible} if it is submultiplicative, symmetric, and satisfies the GRS-condition.
\end{enumerate}  
Let $1<p<\infty$. A submultiplicative weight $\omega$ on $\mathbb{R}^d$ is a \textit{$p$-algebra weight} if it satisfies 
\begin{align} \label{eq:weightalgebra}
    \omega^{-p'}\star\omega^{-p'}\leq\omega^{-p'},
\end{align}
where $\star$ denotes convolution.

Given two weights $\omega$ and $\nu$, $\nu$ is \textit{$\omega$-moderate} if $\nu(x+y)\leq C\nu(x)\omega(y)$ for all $x,y \in \mathbb{R}^d$, where the constant $C$ is independent of $x$ and $y$. Note that if $\omega$ is a submultiplicative weight then it is $\omega$-moderate by definition, and if it is also symmetric, then $\omega\geq1$.  
Some examples of weights on $\mathbb{R}^d$ are 
\begin{enumerate}
	\item[(a)] polynomial type: $\omega_s(x)=(1+|x|)^s$ for $s\geq0$.
	\item[(b)] subexponential type: $\omega(x)=e^{a|x|^b}$ for $a>0, 0<b<1$.
	\item[(c)] mixed form: $\omega(x)=e^{a|x|^b}(1+|x|)^s\log(e+|x|)^t,$ where $a, s, t\geq0$, $0\leq b<1$.
\end{enumerate}
A particular polynomial weight $\omega_s=(1+|x|^2)^\frac{s}{2}$ for $s\geq0$ is of utmost importance in the fields of analysis and has been extensively studied in the literature, mainly with applications in fields of time-frequency analysis, differential equations, frame theory, etc. For $X\subset\mathbb{R}^d$, we consider the weight $\omega$ on $X$ as the restriction of a weight $\omega$ on $\mathbb{R}^d$.


Let $0<p<\infty$, $X$ be a relatively separated subset of $\mathbb{R}^d$, $\omega$ be a weight, and let $B$ be a Banach space. Let $$\ell^p_\omega(X,B)=\left\{x=(x_k)_{k\in X}:x_k \in B, \sum_{k\in X}\|x_k\|^p\omega(k)^p <\infty\right\}.$$
If $0<p<1$, then $\ell^p_\omega(X,B)$ is a $p$-Banach space with the $p$-norm 
$$\|x\|_{\ell^p_\omega(X,B)}=\sum_{k\in X}\|x_k\|^p\omega(k)^p \quad(x \in \ell^p_\omega(X,B))$$ 
and if $1\leq p<\infty$, then $\ell^p_\omega(X,B)$ is a Banach space with the norm 
$$\|x\|_{\ell^p_\omega(X,B)}=\left(\sum_{k\in X}\|x_k\|^p\omega(k)^p \right)^{\frac{1}{p}}\quad(x \in \ell^p_\omega(X,B)).$$ 
Of course, $\ell^\infty_\omega(X,B)$ is a Banach space with the norm $\|x\|_{\ell^\infty_\omega(X,B)} =\sup_{k\in X}\|x_k\|\omega(k)$ for all $x \in \ell^\infty_\omega(X,B)$. Note that the dual of $\ell^p_\omega(X,B)$ is isometrically isomorphic to $\ell^{p'}_{\frac{1}{\omega}}(X,B^\ast)$ for $1\leq p<\infty$, where $B^\ast$ is the dual of $B$. Also, the dual of $\ell^1_\omega(X,B)$ is isometrically isomorphic to $\ell^\infty_{\frac{1}{\omega}}(X,B^\ast)$.

Now, we turn to the space of operators. The spaces of special interest in this paper are the Hilbert space $\ell^2(X,\mathcal{H})$ and the $C^\ast$-algebra $B(\ell^2(X,\mathcal{H}))$.

If $A\in B(\ell^2(X,\mathcal{H}))$, then the \textit{canonical matrix form} of $A$, denoted by $\mathbb{M}(A)$, is given as the matrix $[A_{k,l}]_{k,l\in X}$ with each $A_{k,l}\in B(\mathcal{H})$ unique, and it operates on an element $x=(x_l)_{l\in X}$ in the following manner
\begin{align} \label{def:matrixoperates}
Ax = [A_{k,l}](x) = \left( \sum_{l\in X} A_{k,l} x_l \right)_{k\in X}\quad(x=(x_l)\in \ell^2(X,\mathcal H)).
\end{align}

Conversely, let $A_{k,l}\in B(\mathcal{H})$ for each $k,l\in X$, and let $A=[A_{k,l}]_{k,l\in X}$ be the $B(\mathcal{H})$-valued matrix such that for each $x=(x_l)_{l\in X}\in\ell^2(X,\mathcal{H})$, we have 
\begin{align} \label{eq:domaincondition}
\sum_{k\in X} \left\| \sum_{l\in X} A_{k,l} x_l \right\|^2 < \infty. 
\end{align}
Then $A:\ell^2(X,\mathcal{H})\to\ell^2(X,\mathcal{H})$ defined as above in \eqref{def:matrixoperates} is a bounded linear operator, that is, $A\in B(\ell^2(X,\mathcal{H}))$. So, there is a one-to-one correspondence between the operators $A\in B(\ell^2(X,\mathcal{H}))$ and $B(\mathcal{H})$-valued matrices over $X$ which satisfies \eqref{eq:domaincondition}.

Note that the composition of operators $A,B\in B(\ell^2(X,\mathcal{H}))$ is same as the matrix multiplication of $\mathbb{M}(A)$ and $\mathbb{M}(B)$, that is to say that $\mathbb{M}(AB)=\mathbb{M}(A)\mathbb{M}(B)$. The addition and scalar multiplication also follow the same rule. This implies that the algebra $B(\ell^2(X,\mathcal{H}))$ is isomorphic to the algebra of above mentioned $B(\mathcal{H})$-valued matrices. 

The (matrix) involution for $A=[A_{k,l}]_{k,l\in X}$ is then given by
\begin{align} \label{def:matrixinvolution}
A^\ast=[A^\ast_{k,l}]_{k,l\in X}^t,
\end{align}
where $A_{k,l}^\ast$ is the adjoint of $A_{k,l}\in B(\mathcal{H})$ for $k,l\in X$ and $M^t$ denotes the transpose of a matrix $M$. This establishes the relation with the involution between the above two algebras. We refer the reader to \cite{IJM} where these facts are first mentioned and to \cite{LK,LK1} for detailed proofs.

Lastly, we state some results which will be used to prove our results. First, we state Hulanicki's lemma for quasi algebras in the form we require. It was proved for Banach algebras in \cite{hul} by Hulanicki and for $p$-Banach algebras in \cite{kbhul}.

\begin{theorem} [Hulanicki’s lemma for quasi algebras] \label{thm:Hul}
Let $0<p\leq1$, and let $\mathcal{A}$ be a subalgebra of a symmetric $p$-Banach $\ast$-algebra $\mathcal{B}$ both having same unit. Then the following are equivalent
\begin{enumerate}
    \item $\mathcal{A}$ is inverse-closed in $\mathcal{B}$.
    \item $r_\mathcal{A}(A)=r_\mathcal{B}(A)$ for all $A=A^\ast\in\mathcal{A}$.
    \item $r_\mathcal{A}(A)\leq r_\mathcal{B}(A)$ for all $A=A^\ast\in\mathcal{A}$.
\end{enumerate}
In particular, if one of the above holds, then $\mathcal{A}$ is symmetric.
\end{theorem}

The second result of importance which we require is the multivariate (non-commutative) Banach algebra valued weighted analogue of Wiener's theorem. It traces back to the work of Bochner and Phillips in \cite{bo} which dealt with the non-weighted case. The one variable weighted case is treated in \cite{kb} for all $0<p<\infty$ and multivariate weighted result is obtained in \cite{kbmul} for $0<p\leq1$. Note that the results in \cite{kb,kbmul} deal with larger class of weights while here we require only results for admissible weights and so we state it in the form required. For $1<p<\infty$, the multivariate weighted case with admissible weight follows using same line of arguments.

\begin{theorem}[Vector-valued weighted Wiener's theorem] \label{thm:Wiener}
Let $0<p<\infty$, $\mathcal{A}$ be a unital Banach algebra and let $\omega$ be an admissible weight on $\mathbb{Z}^d$ that is a $p$-algebra weight satisfying $\sum_{n\in\mathbb{Z}^d} \omega(n)^{-p'}<\infty$ for $p>1$. If $f:\mathbb{T}^d\to\mathcal{A}$ is a continuous function such that $\widehat{f}\in\ell^p_\omega(\mathbb{Z}^d,\mathcal{A})$ and $f(t)$ is invertible in $\mathcal{A}$ for all $t\in\mathbb{T}^d$, then $f$ is invertible and $\widehat{f^{-1}}\in\ell^p_\omega(\mathbb{Z}^d,\mathcal{A})$. 
\end{theorem}

Here, $\widehat{f}$ denotes the Fourier transform of $f:\mathbb{T}^d\to\mathcal{A}$ given by
$$\widehat{f}(n)=\frac{1}{2\pi^2}\int_{[0,2\pi]^d}f(e^{ix})e^{-i(n\cdot x)}dx \quad (n\in\mathbb{Z}^d).$$

\section{Schur-type algebras}\label{sec:Schur}
Throughout this section $X$ is a relatively separated subset of $\mathbb{R}^d$ and $\mathcal{S}^p_\omega=\mathcal{S}^p_\omega(X)$. In order to employ the Hulanicki's lemma for quasi algebras, we need to show that $\mathcal{S}^p_\omega$ is indeed a $p$-Banach $\ast$-subalgebra of $B(\ell^2(X,\mathcal{H}))$. We first show that $\mathcal{S}^p_\omega$ is continuously embedded in $B(\ell^2(X,\mathcal{H}))$, and begin with the following lemma which gives the inclusion between Schur-type algebras.

\begin{lemma}
Let $0<p\leq1$, $p\leq q<\infty$, and let $\omega$ be a weight on $\mathbb{R}^d$. Then $\mathcal S_\omega^p\subset \mathcal S_\omega^q$ with 
\begin{align*}
    \|A\|_{\mathcal S_\omega^q}\leq\|A\|_{\mathcal S_\omega^p}^{\frac{q}{p}} \quad \text{for} \quad 0<p\leq q\leq 1, \quad
    \text{and} \quad \|A\|_{\mathcal S_\omega^q}\leq \|A\|_{\mathcal S_\omega^p}^{\frac{1}{p}} \quad \text{for} \quad q>1.
\end{align*}
\end{lemma}
\begin{proof}
The proof follows from the following observations. Let $0<p<q\leq 1$. Then $1\leq \frac{1}{q}<\frac{1}{p}$. This gives $\left(\sum_{k\in X}|x_k|^{\frac{1}{p}}\right)^p\leq \left(\sum_{k\in X}|x_k|^{\frac{1}{q}}\right)^q$ for every complex sequence $(x_k)$. Replacing $x_k$ by $|x_k|^{pq}$, we get 
$$\sum_{k\in X}|x_k|^q \leq \left(\sum_{k\in X}|x_k|^p\right)^{\frac{q}{p}}.$$ 
Taking $q=1$, we have 
$$\sum_{k\in X}|x_k| \leq \left(\sum_{k\in X}|x_k|^p\right)^{\frac{1}{p}}.$$ 
If $q>1$, then 
$$\left(\sum_{k\in X}|x_k|^q\right)^{\frac{1}{q}}\leq \sum_{k\in X}|x_k|\leq \left(\sum_{k\in X}|x_k|^p\right)^{\frac{1}{p}}.$$ 
\end{proof}

\begin{lemma}
Let $0<p<1$, $\omega$ be an admissible weight, and let $\nu$ be $\omega$-moderate weight. Then every $A\in \mathcal S_\omega^p$ defines a bounded operator on $\ell^q_\nu(X,\mathcal H)$ for all $q\geq p$ and $$\|A\|_{B(\ell^q_\nu(X,\mathcal H))}\leq C \|A\|_{B(\ell^p_\omega(X,\mathcal H))}^{\frac{1}{p}},$$ where $C$ is the constant arising from $\nu$ being $\omega$-moderate.
\end{lemma}
\begin{proof}
Let $0<p<q\leq 1$, and let $S^q_\omega=\{f\in\ell^q_\nu(X,\mathcal{H}):\|f\|_{\ell^q_\nu(X,\mathcal{H})}=1\}$ be the unit sphere in $\ell^q_\nu(X,\mathcal{H})$. Then
\begin{eqnarray*}
\|A\|_{B(\ell^q_\nu(X,\mathcal H))} & = & \sup_{f\in S^q_\omega}\sum_{k\in X}\left\|\sum_{l\in X}A_{k,l}f_l\right\|^q \nu(l+k-l)^q\\
& \leq & \sup_{f\in S^q_\omega}\sum_{k\in X}\sum_{l\in X}\|A_{k,l}f_l\|^q \nu(l+k-l)^q\\
& \leq & C^q \sup_{f\in S^q_\omega}\sum_{l\in X}\|f_l\|^q \nu(l)^q\sum_{k\in X}\|A_{k,l}\|^q\omega(k-l)^q\\
& \leq & C^q \sup_{f\in S^q_\omega}\sum_{l\in X}\|f_l\|^q \nu(l)^q\left(\sum_{k\in X}\|A_{k,l}\|^p\omega(k-l)^p\right)^{\frac{q}{p}}\\
& \leq & C^q\|A\|_{B(\ell^p_\omega(X,\mathcal H))}^{\frac{q}{p}}.
\end{eqnarray*}
Thus $\|A\|_{B(\ell^q_\nu(X,\mathcal H))} \leq  C^q\|A\|_{B(\ell^p_\omega(X,\mathcal H))}^{\frac{q}{p}}$ for all $p\leq q \leq 1$. Let $q>1$. Then, by \cite[Lemma 4.4]{KB}, $\|A\|_{B(\ell^q_\nu(X,\mathcal H))}\leq C \|A\|_{\mathcal S_\omega^1}\leq C \|A\|_{\mathcal S_\omega^p}^{\frac{1}{p}}$
\end{proof}

The above lemma ensures that $\mathcal{S}^p_\omega$ is a subalgebra of $B(\ell^2(X,\mathcal{H}))$. Next, we proceed to establish the quasi-norm structure. For the rest of this section, we fix $0<p\leq1$.

\begin{proposition} \label{prop:Spw is algebra}
If $\omega$ is an submultiplicative and symmetric weight, then $\mathcal S_\omega^p$ is a unital $p$-Banach $\ast$-algebra with respect to matrix multiplication and involution \eqref{def:matrixinvolution}.
\end{proposition}
\begin{proof}
It is easy to verify that $\mathcal{S}^p_\omega$ is a $p$-Banach space. We just show that the $p$-norm is submultiplicative. Let $A=[A_{k,l}]$ and $B=[B_{k,l}]$ be in $S^p_\omega$. Then using submultiplicativity of weight and \eqref{eq:ineqp<1}, we get
\begin{align*}
    \sum_{l\in X}\|(AB)_{k,l}\|^p\omega(k-l)^p 
    &= \sum_{l\in X} \left\|\sum_{m \in X} A_{k,m} B_{m,l} \right\|^p \omega(k-m+m-l)^p \\
    &\leq \sum_{l\in X} \sum_{m \in X} \|A_{k,m}\|^p \|B_{m,l}\|^p \omega(k-m)^p\omega(m-l)^p \\
    &\leq \sum_{m \in X} \|A_{k,m}\|^p \omega(k-m)^p \sum_{l\in X} \|B_{m,l}\|^p \omega(m-l)^p \\
    &\leq \|A\|_{\mathcal{S}^p_\omega} \|B\|_{\mathcal{S}^p_\omega}
\end{align*}
which gives
\begin{align*}
    \sup_{k\in X} \sum_{l\in X}\|(AB)_{k,l}\|^p\omega(k-l)^p \leq \|A\|_{\mathcal{S}^p_\omega} \|B\|_{\mathcal{S}^p_\omega}.
\end{align*}
Similarly, it follows that
\begin{align*}
    \sup_{l\in X} \sum_{k\in X}\|(AB)_{k,l}\|^p\omega(k-l)^p \leq \|A\|_{\mathcal{S}^p_\omega} \|B\|_{\mathcal{S}^p_\omega}.
\end{align*}
So, $\|AB\|_{\mathcal{S}^p_\omega}\leq \|A\|_{\mathcal{S}^p_\omega} \|B\|_{\mathcal{S}^p_\omega}$. Thus, $\mathcal{S}^p_\omega$ is a $p$-Banach algebra. Moreover, the fact that $\omega$ is symmetric implies that $\|A^\ast\|_{\mathcal{S}^p_\omega} = \|A\|_{\mathcal{S}^p_\omega}$. This completes the proof.
\end{proof}

As $\mathcal{S}^p_\omega$ is established as a $p$-Banach $\ast$-algebra of $B(\ell^2(X,\mathcal{H}))$, we move towards establishing the inequality for the spectral radius in order to use Hulanicki's lemma. The trick to do is to construct a net of polynomial weights and approximating the spectral radius using the approximation of the $p$-norm.

\begin{lemma}
Let $\omega(x)=\omega_\delta(x)=(1+|x|)^\delta$ for some $\delta \in (0,1]$. Then $r_{\mathcal S_\omega^p}(A)=r_{\mathcal S^p}(A)$ for all $A\in \mathcal S_\omega^p$.
\end{lemma}
\begin{proof}
For $0<\epsilon\leq 1$, let $\mu_\epsilon(x)=(1+\epsilon |x|)^\delta$. Then $\mu_\epsilon$ is a submultiplicative weight and $\mathcal S^p_{\mu_\epsilon}$ is a $p$-Banach algebra. Note that $\mu_\epsilon \leq \omega \leq \epsilon^{-\delta}\mu_\epsilon$. This gives $\|A\|_{\mathcal S_\omega^p}\leq \epsilon^{-p\delta}\|A\|_{\mathcal S_{\mu_\epsilon}}$ and hence $r_{\mathcal S_\omega^p}(A)\leq r_{\mathcal S^p_{\mu_\epsilon}}(A)\leq \|A\|_{\mathcal S_{\mu_\epsilon}^p}$ for all $A \in \mathcal S_\omega^p$.

Let $A\in \mathcal S_\omega^p$. Note that $t^\delta\leq (1+t)^\delta \leq 1+t^\delta$ for $t\geq 0$. Now,
\begin{eqnarray*}
\sup_{k\in X}\sum_{l\in X}\|A_{k,l}\|^p & \leq & \sup_{k\in X}\sum_{l\in X}\|A_{k,l}\|^p(1+\epsilon|k-l|)^{p\delta}\\
& \leq & \sup_{k\in X}\sum_{l\in X}\|A_{k,l}\|^p(1+\epsilon^{p\delta}|k-l|^{p\delta})\\
& \leq & \sup_{k\in X}\sum_{l\in X}\|A_{k,l}\|^p+\epsilon^{p\delta}\sup_{k\in X}\sum_{l\in X}\|A_{k,l}\|^p(1+|k-l|)^{p\delta}.
\end{eqnarray*}
Interchanging the roles of $k$ and $l$, we get
\begin{eqnarray*}
\sup_{l\in X}\sum_{k\in X}\|A_{k,l}\|^p & \leq & \sup_{l\in X}\sum_{k\in X}\|A_{k,l}\|^p(1+\epsilon|k-l|)^{p\delta}\\
& \leq & \sup_{l\in X}\sum_{k\in X}\|A_{k,l}\|^p+\epsilon^{p\delta}\sup_{l\in X}\sum_{k\in X}\|A_{k,l}\|^p(1+|k-l|)^{p\delta}.
\end{eqnarray*}
The above inequalities will give 
$$\|A\|_{\mathcal S_\omega^p}\leq \|A\|_{\mathcal S_{\mu_\epsilon}^p}\leq \|A\|_{\mathcal S^p}+\epsilon^{\delta p}\|A\|_{\mathcal S_\omega^p}.$$ 
Therefore 
$$r_{\mathcal S_\omega^p}(A)\leq \lim_{\epsilon \to 0}\|A\|_{\mathcal S_{\mu_\epsilon}}=\|A\|_{\mathcal S^p}$$ 
and hence $r_{\mathcal S_\omega^p}(A)\leq r_{\mathcal S^p}(A)$ for all $A\in \mathcal S_\omega^p$. Since $\mathcal S_\omega^p\subset \mathcal S^p$, $r_{\mathcal S^p}(A)\leq r_{\mathcal S_\omega^p}(A)$ for all $A\in \mathcal S_\omega^p$.
\end{proof}

\begin{theorem} \label{thm:equivradius}
Let $\omega_\delta(x)=(1+|x|)^\delta$ for some $\delta\in (0,1]$. Then, for all $A=A^\ast \in \mathcal S_{\omega_\delta}^p$, we have
$$r_{\mathcal S_{\omega_\delta}^p}(A)=r_{\mathcal S^p}(A)=r_{B(\ell^2(X,\mathcal H))}(A)^p.$$
\end{theorem}

\begin{proof}
Let $B=B(\ell^2(X,\mathcal H))$. We only need to prove that $r_B(A)^p=r_{\mathcal S^p}(A)$ for all $A=A^\ast \in \mathcal S_{\omega_\delta}^p$. Fix some non-zero $A=A^\ast \in \mathcal S^p_{\omega_\delta}$. Since $A$ is self-adjoint, one has $$\|A\|_{\mathcal S^p}=\sup_{l\in X}\sum_{k\in X}\|A_{k,l}\|^p.$$ So, if $n\in \mathbb N$, then
\begin{eqnarray*}
\|A^{n+1}\|_{\mathcal S^p}\leq \sup_{l\in X}\sum_{\substack{k\in X \\ |k-l|\leq 2^n}}\|[A^{n+1}]_{k,l}\|^p+\sup_{l\in X}\sum_{\substack{k\in X \\ |k-l|> 2^n}}\|[A^{n+1}]_{k,l}\|^p.
\end{eqnarray*}

For $l\in X$, let $B_{2^n}(l)=\{k\in X:|k-l|\leq 2^n\}$. Then there exists $C_1>0$ such that $|B_{2^n}(l)|\leq C_1 2^{nd}$ for all $l \in X$.

Let $l\in X$ be arbitrary. Then
\begin{eqnarray*}
\sum_{\substack{k\in X \\ |k-l|\leq 2^n}}\|[A^{n+1}]_{k,l}\|^p & = & \sum_{\substack{k\in X \\ |k-l|\leq 2^n}}\left\|\sum_{m\in X}[A^n]_{k,m} A_{m,l}\right\|^p\cdot 1\\
& \leq & (C_12^{nd})^{\frac{1}{2}} \left(\sum_{\substack{k\in X \\ |k-l|\leq 2^n}}\left\|\sum_{m\in X}[A^n]_{k,m} A_{m,l}\right\|^{2p}\right)^{\frac{1}{2}}.
\end{eqnarray*}
Take $\epsilon^{(k)}=\|A\|_{B}^{2np}(1+|k|)^{-(d+1)}$ for all $k$. Then there exists $f^{(k)}\in \mathcal H$ such that $$\left\|\sum_{m\in X}[A^n]_{k,m}A_{m,l}\right\|^{2p}-\epsilon^{(k)}<\left\|\sum_{m\in X}[A^n]_{k,m}A_{m,l}f^{(k)}\right\|^{2p}.$$ This gives
\begin{eqnarray*}
&& \sum_{\substack{k\in X \\ |k-l|\leq 2^n}}\left\|\sum_{m\in X}[A^n]_{k,m}A_{m,l}\right\|^{2p}\\
& \leq & \|A\|_B^{2np}\sum_{\substack{k\in X \\ |k-l|\leq 2^n}}(1+|k|)^{-(d+1)}+\sum_{\substack{k\in X \\ |k-l|\leq 2^n}}\left\|\sum_{m\in X}[A^n]_{k,m}A_{m,l}f^{(k)}\right\|^{2p}.
\end{eqnarray*}

Let $t\in \mathbb N$ be such that $2^tp>2$. Then
\begin{eqnarray*}
&& \sum_{\substack{k\in X \\ |k-l|\leq 2^n}}\left\|\sum_{m\in X}[A^n]_{k,m}A_{m,l}f^{(k)}\right\|^{2p}\\
& \leq & \left(\sum_{\substack{k\in X \\ |k-l|\leq 2^n}}\left\|\sum_{m\in X}[A^n]_{k,m}A_{m,l}f^{(k)}\right\|^{2^2p}\right)^{\frac{1}{2}} \left(\sum_{\substack{k\in X \\ |k-l|\leq 2^n}} 1\right)^{\frac{1}{2}}\\
& \leq & \left(\sum_{\substack{k\in X \\ |k-l|\leq 2^n}}\left\|\sum_{m\in X}[A^n]_{k,m}A_{m,l}f^{(k)}\right\|^{2^3p}\right)^{\frac{1}{2^2}} \left(\sum_{\substack{k\in X \\ |k-l|\leq 2^n}} 1\right)^{\frac{1}{2}+\frac{1}{2^2}}\\
& \leq & \left(\sum_{\substack{k\in X \\ |k-l|\leq 2^n}}\left\|\sum_{m\in X}[A^n]_{k,m}A_{m,l}f^{(k)}\right\|^{2^tp}\right)^{\frac{1}{2^{t-1}}} \left(\sum_{\substack{k\in X \\ |k-l|\leq 2^n}} 1\right)^{\frac{1}{2}+\frac{1}{2^2}+\cdots +\frac{1}{2^{t-1}}}\\
& \leq & \left(\sum_{\substack{k\in X \\ |k-l|\leq 2^n}}\left\|\sum_{m\in X}[A^n]_{k,m}A_{m,l}f^{(k)}\right\|^{2}\right)^{p} (C_12^{nd})^{\frac{1}{2}+\frac{1}{2^2}+\cdots +\frac{1}{2^{t-1}}}\\
& \leq & \|A\|_{B}^{2np}\|A\|_{\mathcal S_2^p}^{2p}(C_12^{nd})^{\frac{1}{2}+\frac{1}{2^2}+\cdots +\frac{1}{2^{t-1}}+p}.
\end{eqnarray*}
So,
\begin{eqnarray*}
\sum_{\substack{k\in X \\ |k-l|\leq 2^n}}\left\|\sum_{m\in X}[A^n]_{k,m}A_{m,l}\right\|^{2p} & \leq & C_2 \|A\|_B^{2np}+\|A\|^{2np}_B\|A\|_{\mathcal S_2^p}^{2p}(C_12^{nd})^{\frac{1}{2}+\frac{1}{2^2}+\cdots +\frac{1}{2^{t-1}}+p}.
\end{eqnarray*}
Thus
\begin{eqnarray*}
\sup_{l\in X}\sum_{\substack{k\in X \\ |k-l|\leq 2^n}}\|[A^{n+1}]_{k,l}\|^p & \leq & (C_12^{nd})^{\frac{1}{2}}\left(C_2 \|A\|_B^{2np}+\|A\|^{2np}_B\|A\|_{\mathcal S_2^p}^{2p}(C_12^{nd})^{\frac{1}{2}+\frac{1}{2^2}+\cdots +\frac{1}{2^{t-1}}+p}\right)^{\frac{1}{2}}\\
&\leq & (C_1C_2)^{\frac{1}{2}}2^{\frac{nd}{2}}\|A\|_B^{np}+(C_12^{nd})^{\frac{1}{2}+\frac{1}{2^2}+\cdots +\frac{1}{2^t}+\frac{p}{2}}\|A\|^{np}_B\|A\|_{\mathcal S_2^p}^{p}.
\end{eqnarray*}
For $|k-l|>2^n$, one has $(1+|k-l|)^\delta>(1+2^n)^\delta>2^{n\delta}$ and hence $1<2^{-n\delta}\omega_\delta(k-l)$. Therefore
\begin{eqnarray*}
\sup_{l\in X}\sum_{\substack{k\in X \\ |k-l|>2^n}} \|[A^{n+1}]_{k,l}\|^p & \leq & 2^{-pn\delta}\sup_{l\in X}\sum_{\substack{k\in X \\ |k-l|>2^n}}\|[A^{n+1}]_{k,l}\|^p\omega_\delta(k-l)\\
& \leq & 2^{-pn\delta}\|A^{n+1}\|_{\mathcal S^p_{\omega_\delta}}.
\end{eqnarray*}
Take $\alpha=\frac{1}{2}+\frac{1}{2^2}+\cdots +\frac{1}{2^t}+\frac{p}{2}$. Then we have
\begin{align*}
\|A^{n+1}\|_{\mathcal S^p}^{\frac{1}{n+1}}
 \leq  (C_1C_2)^{\frac{1}{2(n+1)}}2^{\frac{nd}{2(n+1)}}\|A\|_B^{\frac{np}{n+1}} +(C_12^{nd})^{\frac{\alpha}{n+1}}\|A\|^{\frac{np}{n+1}}_B\|A\|_{\mathcal S_2^p}^{\frac{p}{n+1}}+2^{-\frac{pn\delta}{n+1}}\|A^{n+1}\|_{\mathcal S^p_{\omega_\delta}}^{\frac{1}{n+1}}.
\end{align*}
This gives
\begin{eqnarray*}
r_{\mathcal S^p}(A) &\leq & (2^{\frac{d}{2}}+2^{\alpha d})\|A\|_B^p+2^{-p\delta}r_{\mathcal S^p_{\omega_\delta}}(A)\\
&=& (2^{\frac{d}{2}}+2^{\alpha d})\|A\|_B^p+2^{-p\delta}r_{\mathcal S^p}(A).
\end{eqnarray*}
Therefore $$r_{\mathcal S^p}(A)\leq (1-2^{-p\delta})^{-1}(2^{\frac{d}{2}}+2^{\alpha d})\|A\|_B^p.$$ The above holds for all self-adjoint $A \in \mathcal S_{\omega_\delta}^p$. So, $$r_{\mathcal S^p}(A)=r_{\mathcal S^p}(A^n)^{\frac{1}{n}}\leq ((1-2^{-p\delta})^{-1}(2^{\frac{d}{2}}+2^{\alpha d}))^{\frac{1}{n}}\|A^n\|_B^{\frac{p}{n}}$$ and hence $r_{\mathcal S^p}(A)\leq r_B(A)^p$ for all $A=A^\ast \in \mathcal S_{\omega_\delta}^p$.
\end{proof}

We require the following technical lemma for weights.
\begin{lemma} \cite[Lemma 8]{GroLei} \label{lem:8}
Let $\omega$ be an unbounded admissible weight $\omega$. Then there is a sequence of admissible weights $\omega_n$ satisfying 
\begin{enumerate}
    \item $\omega_{n+1}\leq\omega_n\leq\omega$ for all $n\in\mathbb{N}$;
    \item there are constants $c_n>0$ such that $\omega\leq c_n\omega_n$ for all $n\in\mathbb{N}$; and 
    \item $\displaystyle \lim_{n\to\infty} \omega_n = 1$ uniformly on compact sets of $\mathbb{R}^d$.
\end{enumerate}
\end{lemma}

Next, we state a lemma whose proof goes along the same line of the proof of \cite[Lemma 9 (a)]{GroLei} with minor modification.
\begin{lemma} \label{lem:9}
Let $\omega$ be an admissible weight on $\mathbb{R}^d$ which also satisfies the weak growth condition. Then, for all $A=A^\ast \in \mathcal{S}^p_\omega$ and weights $\omega_n$ as in Lemma \ref{lem:8},
$$\lim_{n\to\infty} \|A\|_{\mathcal{S}^p_{\omega_n}} = \|A\|_{\mathcal{S}^p_\omega} \quad \text{and} \quad r_{\mathcal{S}^p_{\omega_n}}(A)=r_{\mathcal{S}^p_\omega}(A)=\|A\|_{op}^p.$$
\end{lemma}

With last two lemmas and Theorem \ref{thm:equivradius} at disposal, we are now ready for the main theorem which follows from the following theorem. We follow the technique used in \cite[Lemma 9 (b)]{GroLei}.

\begin{theorem} \label{thm:main1}
Let $0<p\leq1$, $\omega$ be an admissible weight satisfying the weak growth condition \eqref{def:weakgrowthweight}, and let $A=A^\ast\in\mathcal{S}^p_{\omega}$. Then $$r_{\mathcal{S}^p_{\omega}}(A)=\|A\|_{op}^p.$$ 
\end{theorem}
\begin{proof}
Using the equivalence of the weights $\omega_n$ from Lemma \ref{lem:8}, for all $A\in\mathcal{S}^p_\omega$ and $m,n\in\mathbb{N}$, we have
\begin{align*}
    r_{\mathcal{S}^p_\omega}(A)^m = r_{\mathcal{S}^p_\omega}(A^m) = r_{\mathcal{S}^p_{\omega_n}} (A^m) \leq \|A^m\|_{\mathcal{S}^p_{\omega_n}}.
\end{align*}
Then, for all $A=A^\ast\in\mathcal{S}^p_\omega$ and $m\in\mathbb{N}$, Lemma \ref{lem:9} implies that
\begin{align*}
    r_{\mathcal{S}^p_\omega}(A)^m \leq \lim_{n\to\infty} \|A^m\|_{\mathcal{S}^p_{\omega_n}} = \|A^m\|_{\mathcal{S}^p}.
\end{align*}
This gives
\begin{align*}
    r_{\mathcal{S}^p_\omega}(A) \leq \lim_{m\to\infty}  \|A^m\|^\frac{1}{m}_{\mathcal{S}^p} = r_{\mathcal{S}^p}(A) \quad \text{for all} \quad A=A^\ast\in\mathcal{S}^p_\omega.
\end{align*}
Now, using that fact that $\omega$ satisfies \eqref{def:weakgrowthweight}, we get $\mathcal{S}^p_\omega \subset \mathcal{S}^p_{\omega_\delta}$ and so, by Theorem \ref{thm:equivradius}, for all $A=A^\ast\in\mathcal{S}^p_\omega$, we have
\begin{align*}
    r_{\mathcal{S}^p_\omega}(A) \leq r_{\mathcal{S}^p}(A) =r_{B(\ell^2(X,\mathcal H))}(A)^p.
\end{align*}
The reverse inequality follows from the fact that $\mathcal{S}^p_\omega\subset B(\ell^2(X,\mathcal{H}))$. 
\end{proof}

\begin{proof} [Proof of Theorem \ref{thm:main Schur}]
The proof follows from Theorem \ref{thm:main1} combined with Theorem \ref{thm:Hul}.
\end{proof}

\section{BGS-type algebras} \label{sec:BGS}

In this section, we deal with the weighted $p$-normed $B(\mathcal{H})$-valued Baskakov-Gohberg-Sj\"ostrand type algebra and for it we take $X=\mathbb{Z}^d$. Recall that
\begin{align*}
\mathcal{C}^p_\omega = \mathcal{C}^p_\omega(\mathbb{Z}^d) = \left\{A=[A_{k,l}]_{k,l}\in\mathbb{Z}^d : A_{k,l} \in B(\mathcal{H}), |A| = \sum_{l\in\mathbb{Z}^d} \left(\sup_{k\in\mathbb{Z}^d} \|A_{k,k-l}\|\right)^p \omega(l)^p<\infty \right\}.
\end{align*}

\begin{proposition} \label{prop:Cpw is algebra}
Let $\omega$ be a submultiplicative weight on $\mathbb{Z}^d$.
\begin{enumerate}
    \item If $0<p\leq 1$, then $\mathcal{C}^p_\omega$ is a $p$-Banach algebra with the $p$-norm $\|\cdot\|_{\mathcal{C}^p_\omega}=|\cdot|$. 
    \item If $1<p<\infty$ and if $\omega$ satisfies $\omega^{-p'}\star\omega^{-p'}\leq\omega^{-p'}$, then $\mathcal{C}^p_\omega$ is a Banach algebra with the norm $\|\cdot\|_{\mathcal{C}^p_\omega}=|\cdot|^\frac{1}{p}$. 
\end{enumerate}
\end{proposition}
\begin{proof}
Again, we just show the submultiplicativity of the ($p$-)norm as the rest is well established in literature. The case of $0<p\leq1$ can be shown using techniques used in Proposition~\ref{prop:Spw is algebra}. Now, let $1<p<\infty$, and let $A,B \in C^p_\omega$. Then for each $l\in\mathbb{Z}^d$, using H\"older's inequality and \eqref{eq:weightalgebra}, we get
\begin{align*}
    &\sup_{k\in\mathbb{Z}^d} \|(AB)_{k,k-l}\| \\ 
    =& \sup_{k\in\mathbb{Z}^d} \left\|\sum_{m \in \mathbb{Z}^d} A_{k,m} B_{m,k-l} \right\| \\
    \leq&  \sup_{k\in\mathbb{Z}^d} \sum_{m \in \mathbb{Z}^d} \|A_{k,m}\| \|B_{m,k-l}\|    \\
    \leq&  \sum_{m \in \mathbb{Z}^d} \sup_{k\in\mathbb{Z}^d} \|A_{k,m}\| \omega(k-m) \|B_{m,k-l}\|\omega(m-k+l) \frac{1}{\omega(k-m) \omega(m-k+l)}   \\
    \leq&  \sum_{m \in \mathbb{Z}^d} \sup_{k\in\mathbb{Z}^d} \|A_{k,m}\| \omega(k-m) \|B_{m,k-l}\|\omega(m-k+l) \frac{1}{\omega(k-m) \omega(m-k+l)}    \\
    \leq& \left( \sum_{m \in \mathbb{Z}^d} \left( \sup_{k\in\mathbb{Z}^d} \|A_{k,m}\| \omega(k-m) \|B_{m,k-l}\|\omega(m-k+l) \right)^p \right)^\frac{1}{p} \\ 
    & \cdot \left( \sum_{m \in \mathbb{Z}^d} \frac{1}{\omega(k-m)^{p'} \omega(m-k+l)^{p'}} \right)^\frac{1}{p'}    \\
    \leq& \left( \sum_{m \in \mathbb{Z}^d} \left( \sup_{k\in\mathbb{Z}^d} \|A_{k,m}\| \omega(k-m) \|B_{m,k-l}\|\omega(m-k+l) \right)^p \right)^\frac{1}{p}  \left( \omega^{-p'}\star\omega^{-p'}(l) \right)^\frac{1}{p'} \\
    \leq& \left( \sum_{m \in \mathbb{Z}^d} \left( \sup_{k\in\mathbb{Z}^d} \|A_{k,m}\| \omega(k-m) \|B_{m,k-l}\|\omega(m-k+l) \right)^p \right)^\frac{1}{p}  \omega^{-1}(l) .
\end{align*}
This gives
\begin{align*}
    \|AB\|_{\mathcal{C}^p_\omega}^p 
    &= \sum_{l\in\mathbb{Z}^d} \left(\sup_{k\in\mathbb{Z}^d} \|(AB)_{k,k-l}\|\right)^p \omega(l)^p \\
    &\leq \sum_{l\in\mathbb{Z}^d} \sum_{m \in \mathbb{Z}^d} \left( \sup_{k\in\mathbb{Z}^d} \|A_{k,m}\| \omega(k-m) \right)^p \left( \sup_{k\in\mathbb{Z}^d} \|B_{m,k-l}\|\omega(m-k+l) \right)^p \omega^{-p}(l) \omega(l)^p \\
    &\leq \|A\|_{\mathcal{C}^p_\omega}^p \|B\|_{\mathcal{C}^p_\omega}^p
\end{align*}
This completes the proof.
\end{proof}

Note that $\mathcal{C}^p_\omega$ is isometrically isomorphic to $\ell^p_\omega(\mathbb{Z}^d,\ell^\infty(\mathbb{Z}^d,B(\mathcal{H})))$ as the ($p$-)norm on it can be written as 
\begin{align} \label{eq:Cvnorm2}
    \|A\|_{\mathcal{C}^p_\omega} = \left\| \left( \|A_{k,k-l}\| )_{k\in\mathbb{Z}^d} \|_{\ell^\infty(\mathbb{Z}^d)} \right)_{l\in\mathbb{Z}^d}  \right\|_{\ell^p_\omega(\mathbb{Z}^d)}.
\end{align}

Let $A=[A_{k,l}]_{k,l\in\mathbb{Z}^d}\in\mathcal{C}^p_\omega$. Define $d_A:\mathbb{Z}^d\to[0,\infty)$ by
\begin{align} \label{def:d_A}
    d_A(l)=\sup_{k\in\mathbb{Z}^d} \|A_{k,k-l}\| \quad (l\in\mathbb{Z}^d).
\end{align}
This map $d_A$ is used to characterize the off-diagonal decay of the matrices. It is easy to see that $d_A\in\ell^p_\omega(\mathbb{Z}^d)$ and $\|d_A\|_{\ell^p_\omega}=\|A\|_{\mathcal{C}^p_\omega}$.

\begin{proposition} \label{prop:Cvbddlinop}
Let $0<p\leq1$, and let $\nu$ be a $\omega$-moderate weight. Then each $A\in\mathcal{C}^p_\omega$ defines a bounded linear map on $B(\ell^q_\nu(\mathbb{Z}^d,\mathcal{H}))$ for all $p\leq q \leq\infty$. 
\end{proposition}
\begin{proof}
The case of $p=1$ is shown in \cite[Prop. 6.1]{KB}. If $p<1$, then $\mathcal{C}^p_\omega\subset\mathcal{C}^1_\omega$ and so the case for $1\leq q\leq\infty$ follows from the case of $p=1$. So, it just remains to show for $p<1$ and $p\leq q<1$. Let $x=(x_l)\in\ell^q_\nu(\mathbb{Z}^d,\mathcal{H})$. For $l\in\mathbb{Z}^d$, define $c_l=\|x_l\|$. Then $c=(c_l)\in\ell^q_\nu(\mathbb{Z}^d)$. Keeping up with the notation in \eqref{def:d_A}, we have
\begin{align*}
    \|[Ax]_k\| \leq \sum_{l\in\mathbb{Z}^d} \|A_{k,l}\| \|x_l\| \leq d_A(k-l) c_l = (d_A\star c)(k) \quad (k\in\mathbb{Z}^d).
\end{align*}
Now, using the fact that $\nu$ is $\omega$-moderate and $p\leq q \leq 1$, we get 
\begin{align*}
    \|Ax\|_{\ell^q_\nu(\mathbb{Z}^d,\mathcal{H})} = \|(\|[Ax]\|)_k\|_{\ell^q_\nu(\mathbb{Z}^d)} 
    &\leq \|d_A\star c\|_{\ell^q_\nu(\mathbb{Z}^d)} \\
    &\leq C \|d_A\|_{\ell^q_\omega(\mathbb{Z}^d)} \|c\|_{\ell^q_\nu(\mathbb{Z}^d)} \\
    &\leq C \|d_A\|_{\ell^p_\omega(\mathbb{Z}^d)} \|c\|_{\ell^q_\nu(\mathbb{Z}^d)}\\
    & = C \|A\|_{\mathcal{C}^p_\omega} \|x\|_{\ell^q_\nu(\mathbb{Z}^d,\mathcal{H})},
\end{align*}
where the constant $C$ is from the relation between $\nu$ and $\omega.$
\end{proof}

\begin{proposition} \label{prop:Cvbddlinop>1}
Let $1<p<\infty$, $\omega$ be a $p$-algebra such that $\sum_{n\in\mathbb{Z}^d} \omega(n)^{-p'}<\infty$, and let $\nu$ be a $\omega$-moderate weight. Then each $A\in\mathcal{C}^p_\omega$ defines a bounded linear map on $B(\ell^q_\nu(\mathbb{Z}^d,\mathcal{H}))$ for all $1\leq q \leq\infty$. 
\end{proposition}
\begin{proof}
The proof follows using arguments as in above Proposition \ref{prop:Cvbddlinop} with minor modifications.
\end{proof}

The above Propositions \ref{prop:Cvbddlinop} and \ref{prop:Cvbddlinop>1} can be proved without using the function $d_A$ as done in Proposition \ref{prop:Cpw is algebra}, but is done so to justify the other name of Baskakov-Gohberg-Sj\"ostrand algebra which is \emph{convolution dominated matrices}.

Next, we state a lemma which characterizes the refined algebra structure of $\mathcal{C}^p_\omega$. Its proof is straightforward and similar to the one given in \cite[Prop. 6.2]{KB} for $p=1$.

\begin{lemma} \label{lem:Cvalgebra}
Let $\omega$ be a symmetric submultiplicative weight on $\mathbb{Z}^d$.
\begin{enumerate}
    \item If $0<p\leq1$, then $\mathcal{C}^p_\omega$ is a unital $p$-Banach $\ast$-algebra.
    \item If $1<p<\infty$ and $\omega$ is a $p$-algebra weight satisfying $\sum_{n\in\mathbb{Z}^d} \omega(n)^{-p'}<\infty$, then $\mathcal{C}^p_\omega$ is a unital Banach $\ast$-algebra.
\end{enumerate}
The involution is as in \eqref{def:matrixinvolution}. In particular, $\mathcal{C}^p_\omega \subset B(\ell^2(\mathbb{Z}^d,\mathcal{H}))$.
\end{lemma}

Now we are ready for the proof of Theorem \ref{thm:main Cv}. Its proof is obtained using (non-commutative) $p$-Banach algebra valued weighted analogue of Wiener's theorem. A similar result constructing Fourier series using resolution of identity for $B(\mathcal{X})$-valued function is obtained by Baskakov in \cite{Ba} and its $p$-power analogues are obtained in \cite{kbwp}, where $\mathcal{X}$ is a Banach space. The case of $p=1$ is obtained in \cite{KB} and we follow its line of arguments. 

\begin{proof} [Proof of Theorem \ref{thm:main Cv}]
    For $t\in\mathbb{R^d}$, by $M_t\in B(\ell^2(\mathbb{Z}^d,\mathcal{H}))$ denote the \textit{modulation operator} which is defined as a diagonal matrix by 
\begin{align*}
    M_t=\text{diag}[e^{2\pi ik\cdot t} \mathcal{I}_{B(\mathcal{H})}]_{k\in\mathbb{Z}^d}.
\end{align*}
Given $x=(x_l)_l\in\ell^2(\mathbb{Z}^d,\mathcal{H})$, $M_t$ operates on it as $M_t(x)=(e^{2\pi il\cdot t}x_l)_{l\in\mathbb{Z}^d}$. Next, we define a map $M:\mathbb{T}^d\to B(\ell^2(\mathbb{Z}^d,\mathcal{H}))$ by $M(e^{it})=M_t$. Then $M$ is a unitary representation of $\mathbb{T}^d$, the dual group of $\mathbb{Z}^d$. For $A=[A]_{k,l\in\mathbb{Z}^d}\in B(\ell^2(\mathbb{Z}^d,\mathcal{H}))$, define a 1-periodic $B(\ell^2(\mathbb{Z}^d,\mathcal{H}))$-valued function $f_A$ by
\begin{align*}
    f_A(t)=M_tAM_{-t} \quad (t\in\mathbb{R}^d).
\end{align*}
Then $f_A$ can be viewed as a $B(\ell^2(\mathbb{Z}^d,\mathcal{H}))$-valued function on $\mathbb{T}^d$. The canonical matrix representation of $f_A(t)$ is given by
\begin{align*}
    \mathbb{M}(f_A(t))=[A_{k,l} e^{2\pi i(k-l)\cdot t}]_{k,l\in\mathbb{Z}^d}.
\end{align*}
Observe that if $A\in\mathcal{C}^p_\omega$, then $f_A$ has operator valued Fourier series representation given by
\begin{align} \label{def:Fourier series of F_A}
    f_A(t)=\sum_{n\in\mathbb{Z}^d} D_A(n) e^{2\pi in\cdot t} \quad (t\in \mathbb{R}^d),
\end{align}
where $D_A(n)$ is the $n$-th side diagonal of $A$, that is,
\begin{align} \label{def:D_A}
    [D_A(n)]_{k,l}=\begin{cases}
        A_{k,l}, & \text{if} \ \ l=k-n \\ 0, & \text{if} \ \ l\neq k-n
    \end{cases}.
\end{align}
The above means that the Fourier coefficients of $f_A$ coincide with $D_A$, that is, $\widehat f_A(n)=D_A(n)$ for all $n\in\mathbb{Z}^d$.

If $0<p\leq1$, then using $\omega\geq1$ along with \eqref{eq:ineqp<1}, we get 
\begin{align*}
    \|f_A(t)\|_{B(\ell^2(\mathbb{Z}^d,\mathcal{H}))} \leq \sum_{n\in\mathbb{Z}^d} \|D_A(n)\|_{B(\ell^2(\mathbb{Z}^d,\mathcal{H}))} = \sum_{n\in\mathbb{Z}^d} d_A(n) \leq \left( \sum_{n\in\mathbb{Z}^d} d_A(n)^p \omega(n)^p \right)^\frac{1}{p} 
    = \|A\|_{\mathcal{C}^p_\omega}^\frac{1}{p},
\end{align*}
and from \eqref{def:d_A}, \eqref{def:Fourier series of F_A} and \eqref{def:D_A}, it follows that
\begin{align*}
    \|\widehat f_A\|_{\ell^p_\omega(\mathbb{Z}^d,B(\ell^2(\mathbb{Z}^d,\mathcal{H})))} 
    = \sum_{n\in\mathbb{Z}^d} \|D_A(n)\|^p_{B(\ell^2(\mathbb{Z}^d,\mathcal{H}))} \omega(n)^p 
    = \sum_{n\in\mathbb{Z}^d} d_A(n)^p \omega(n)^p 
     = \|A\|_{\mathcal{C}^p_\omega}.
\end{align*}
Also, if $1<p<\infty$, then using H\"olders inequality and the fact that $\omega$ is a $p$-algebra weight, we get
\begin{align*}
\|f_A(t)\|_{B(\ell^2(\mathbb{Z}^d,\mathcal{H}))} &\leq \sum_{n\in\mathbb{Z}^d} \|D_A(n)\|_{B(\ell^2(\mathbb{Z}^d,\mathcal{H}))} \\
    &= \sum_{n\in\mathbb{Z}^d} d_A(n) \omega(n) \omega(n)^{-1} \\ 
    &\leq \left( \sum_{n\in\mathbb{Z}^d} d_A(n)^p \omega(n)^p \right)^\frac{1}{p} \left( \sum_{n\in\mathbb{Z}^d} \omega(n)^{-p'}  \right)^\frac{1}{p'} \\
    &\leq \|A\|_{\mathcal{C}^p_\omega}^\frac{1}{p} \left( \sum_{n\in\mathbb{Z}^d} \omega(n)^{-p'}  \right)^\frac{1}{p'}.
\end{align*}
and again from \eqref{def:d_A}, \eqref{def:Fourier series of F_A} and \eqref{def:D_A}, we have
\begin{align*}
    \|\widehat f_A\|^p_{\ell^p_\omega(\mathbb{Z}^d,B(\ell^2(\mathbb{Z}^d,\mathcal{H})))} 
    = \sum_{n\in\mathbb{Z}^d} \|D_A(n)\|^p_{B(\ell^2(\mathbb{Z}^d,\mathcal{H}))} \omega(n)^p 
    = \sum_{n\in\mathbb{Z}^d} d_A(n)^p \omega(n)^p 
    = \|A\|_{\mathcal{C}^p_\omega}^p.
\end{align*}
Thus, in both the cases, it follows that the Fourier series of $f_A$ converges absolutely ensuring $f_A$ is well defined and $f_A\in\ell^p_\omega(\mathbb{Z}^d,B(\ell^2(\mathbb{Z}^d,\mathcal{H})))$ provided $A\in\mathcal{C}^p_\omega$.

Now, let $A\in \mathcal{C}^p_\omega$ be such that $A^{-1}\in B(\ell^2(\mathbb{Z}^d,\mathcal{H}))$. Then, for all $t\in\mathbb{T}^d$, $f_A(t)$ is invertible in $B(\ell^2(\mathbb{Z}^d,\mathcal{H}))$ with inverse given by $f_A(t)^{-1}=f_{A^{-1}}(t)=M_tA^{-1}M_{-t}$. Taking $\mathcal{A}=B(\ell^2(\mathbb{Z}^d,\mathcal{H}))$ in Theorem \ref{thm:Wiener}, it follows that $f_A$ is invertible with $(f_A)^{-1}\in\ell^p_\omega(\mathbb{Z}^d,B(\ell^2(\mathbb{Z}^d,\mathcal{H})))$. Also, observe that $(f_A)^{-1}$ is of the form 
\begin{align*}
    (f_A)^{-1}(t)=f_{A^{-1}}(t)=\sum_{n\in\mathbb{Z}^d} D_{A^{-1}}(n) e^{2\pi in\cdot t} \quad (t\in\mathbb{R}^d).
\end{align*}
This with the arguments from above gives $A^{-1}\in\mathcal{C}^p_\omega$. This completes the proof.
\end{proof}

\begin{remark}
Theorem \ref{thm:main Cv} is not true if $\omega$ is not admissible. We shall give an example for $d=1$. Consider the weight $\omega(n)=e^{|n|}$ $(n\in\mathbb{Z})$ that does not satisfy the GRS-condition as $\displaystyle \lim_{n\to\infty} \omega(n)^\frac{1}{n}=e$ and therefore not an admissible weight. Note that the function $f(z)=2-z$ for $z\in\mathbb{T}$ is continuous and nowhere zero function, which implies that $f^{-1}$ exists and it is a continuous function on $\mathbb{T}$. Also, $\widehat{f}\in\ell^1_\omega(\mathbb{Z})$ but $\widehat{f^{-1}}$ is not in $\ell^1_\omega(\mathbb{Z})$. This then gives that $\mathcal{C}^1_\omega$ is not inverse-closed in $B(\ell^2(\mathbb{Z},\mathcal{H}))$.     
\end{remark}

\begin{remark}
The condition that the weight is symmetric can be dropped for BGS-type algebra as it is not required in Theorem \ref{thm:Wiener}, while it is necessary for Schur-type algebra as Hulanicki's lemma Theorem \ref{thm:Hul} demands it.
\end{remark}

In view of above lemmas, we have the following corollary.

\begin{corollary} 
Let $0<p<\infty$, and let $\omega$ be a submultiplicative weight on $\mathbb{Z}^d$ such that for $p>1$, $\omega$ satisfies $\omega^{-p'}\star\omega^{-p'}\leq\omega^{-p'}$ and $\sum_{n\in\mathbb{Z}^d} \omega(n)^{-p'}<\infty$. Then the algebra $\mathcal{C}^p_\omega$ is symmetric if and only if $\omega$ satisfies the GRS-condition.
\end{corollary}

\begin{remark}
Note that Theorem \ref{thm:main Cv}  and above corollary even follows by replacing the Hilbert space $\mathcal{H}$ by some unital Banach space $\mathcal{A}$ and in that case $\ell^2(X,\mathcal{A})$ will be just a Banach space and not a Hilbert space. But it causes no problem as the Wiener's theorem (Theorem \ref{thm:Wiener}) is true even for a unital Banach algebra.
\end{remark}

\section{Gel'fand space of $\ell^p_\omega(\mathbb{Z^\mathbb{N}})$} \label{sec:Gelfand space}
We briefly give the definition of the Gel'fand space. Let $\mathcal{X}$ be a commutative $p$-Banach algebra, and let $\Delta(\mathcal{X})$ be the collection of all nonzero multiplicative linear maps $\varphi:\mathcal{A}\to\mathbb{C}$. For each $a \in \mathcal{X}$, define $\widehat a:\Delta(\mathcal{X})\to \mathbb{C}$ by $\widehat a(\varphi)=\varphi(a)\;(\varphi\in \Delta(\mathcal{X}))$. The smallest topology on $\Delta(\mathcal{X})$ for which $\widehat a$ is continuous, for all $a\in \mathcal{X}$, is the \emph{Gel'fand topology} on $\Delta(\mathcal{X})$, and $\Delta(\mathcal{X})$ with the Gel'fand topology is the \emph{Gel'fand space} of $\mathcal{X}$. For more details, refer to \cite{GRS, ze}.

Let $\mathbb{Z^N}$ be the space of all finitely supported sequences of integers. An element $\alpha$ of $\mathbb{Z}^\mathbb{N}$ is of the form 
\begin{align*}
    \alpha=(\alpha_1,\alpha_2,\dots,\alpha_n,0,0,0,\dots), \quad \text{where} \quad n\in\mathbb{N} \quad \text{and} \quad \alpha_i\in\mathbb{Z} \quad (1\leq i \leq n).
\end{align*} 
For each nonzero $\alpha\in\mathbb{Z}^\mathbb{N}$, there is $n_\alpha\in\mathbb{N}$ such that $\alpha_{n_\alpha}\neq0$ and $\alpha_k=0$ for all $k>n_\alpha$.

A weight on $\mathbb{Z^N}$ is a map $\omega:\mathbb{Z^N}\to[1,\infty)$ that satisfies 
$$\omega(\alpha+\beta)\leq\omega(\alpha) + \omega(\beta) \quad (\alpha,\beta\in\mathbb{Z^N}).$$
A weight on $\mathbb{Z^N}$ is \textit{admissible} if it satisfies the GRS-condition: $\displaystyle \lim_{n\to\infty}\omega(n\alpha)^\frac{1}{n}=1$ for all $\alpha\in\mathbb{Z^N}$. We do not assume $\omega$ to be symmetric here.

For $0<p\leq1$ and a weight $\omega$ on $\mathbb{Z^N}$, let
$$\ell^p_\omega(\mathbb{Z^N})=\ell^p_\omega(\mathbb{Z^N},\mathbb{C})=\left\{f=(f_\alpha)_{\alpha\in \mathbb{Z^N}}:f_\alpha \in \mathbb{C}, \|f\|_{\ell^p_\omega}=\sum_{\alpha\in \mathbb{Z^N}}\|f_\alpha\|^p\omega(\alpha)^p <\infty\right\}.$$
Then $\ell^p_\omega(\mathbb{Z^N})$ is a unital $p$-Banach algebra with the $p$-norm $\|\cdot\|_{\ell^p_\omega}$ and convolution defined as
$$(f\star g)(\alpha)=\sum_{\beta\in\mathbb{Z^N}} f(\alpha-\beta)g(\beta)  \quad (\alpha\in \mathbb{Z^N}, \,\, f,g\in\ell^p_\omega(\mathbb{Z^N})).$$

\subsection{Identification of Gel'fand space}
In \cite{AT}, it is claimed that if $\omega$ is a weight on $\mathbb{Z}^\mathbb{N}$, then the Gel'fand space of $\ell^1_\omega(\mathbb{Z}^\mathbb{N})$ is homeomorphic to the countable product of annulus $\Gamma(\rho_{1,k},\rho_{2,k})=\{z\in\mathbb{C}: \rho_{1,k}\leq|z|\leq\rho_{2,k}\}$, where
\begin{align*}
    \rho_{1,k} = \sup_{\alpha_k<0} \omega(\alpha_k)^\frac{1}{\alpha_k} \quad \text{and} \quad 
    \rho_{2,k} = \inf_{\alpha_k>0} \omega(\alpha_k)^\frac{1}{\alpha_k}
\end{align*}
with $\omega(\alpha_k)=\omega(0,0,0,\dots,0,\alpha_k,0,\dots)$ and $\alpha_k$ at $k^{th}$ position. Here, we show that it is not true by giving an example. Before doing so, we first give an identification of $\Delta(\ell^p_\omega(\mathbb{Z}^\mathbb{N}))$, for $0<p\leq1$, with a subset of $(\mathbb{C}^\times)^\infty$, the countable product of $\mathbb{C}^\times=\mathbb{C}\setminus\{0\}$. 

For $\alpha=(\alpha_1,\alpha_2,\alpha_3,\dots)\in\mathbb{Z}^\mathbb{N}$ and $z=(z_1,z_2,z_3,\dots)\in(\mathbb{C}^\times)^\infty$, define $$z^\alpha=z_1^{\alpha_1}z_2^{\alpha_2}z_3^{\alpha_3}\cdots.$$ 
Since $\alpha$ has finite length, this product is well defined; indeed,
$$z^\alpha=z_1^{\alpha_1}z_2^{\alpha_2}z_3^{\alpha_3} \cdots z_{n_\alpha}^{\alpha_{n_\alpha}} \quad \text{and} \quad |z^\alpha|=|z_1|^{\alpha_1} |z_2|^{\alpha_2} |z_3|^{\alpha_3} \cdots |z_{n_\alpha}|^{\alpha_{n_\alpha}}.$$

\begin{theorem}
Let $0<p\leq 1$, and let $\omega$ be a weight on $\mathbb{Z}^\mathbb{N}$. Then $\Delta(\ell^p_\omega(\mathbb{Z}^\mathbb{N}))$ is homeomorphic to the set 
\begin{align*}
    \Gamma_\omega=\left\{z\in\mathbb{(C^\times)^\infty}: \prod_{i=1}^{n_\alpha} |z_i|^{\alpha_i} \leq \omega(\alpha) \,\, \text{for all} \,\, \alpha\in \mathbb{Z}^\mathbb{N}\right\}.
\end{align*}
\end{theorem}
\begin{proof}
For $i\in\mathbb{N}$, let $e_i$ be an element of $\mathbb{Z^N}$ with $1$ at $i^{th}$ position and zero elsewhere, and for $\alpha=(\alpha_1,\alpha_2,\dots,\alpha_{n_\alpha},0,0,\dots)\in\mathbb{Z^N}$, let $\delta_\alpha\in \ell^p_\omega(\mathbb{Z}^\mathbb{N})$ be the function defined as $\delta_\alpha(\alpha)=1$ and zero otherwise. Observe that 
\begin{align*}
    \delta_\alpha = \delta_{(1,1,1,\dots)}^\alpha = \delta_{e_1}^{\alpha_1} \star \delta_{e_2}^{\alpha_2} \star \cdots \star \delta_{e_{n_\alpha}}^{\alpha_{n_\alpha}}.
\end{align*}

Let $\varphi\in\Delta(\ell^p_\omega(\mathbb{Z}^\mathbb{N}))$. For each $i\in\mathbb{N}$, let $z_i=\varphi(\delta_{e_i})\in\mathbb{C}$. Since $\delta_\mathbf{0}$, for $\mathbf{0}=(0,0,0,\dots)$, is the unit element of $\ell^p_\omega(\mathbb{Z}^\mathbb{N})$, $1=\phi(\delta_\mathbf{0})=\phi(\delta_{e_i}\star\delta_{-e_i})=\phi(\delta_{e_i}) \phi(\delta_{-e_i})$ for all $i\in\mathbb{N}$. So, $z=(z_1,z_2,z_3,\dots)\in(\mathbb{C}^\times)^\infty$.
Now, let $f\in \ell^p_\omega(\mathbb{Z}^\mathbb{N})$. Then $f$ can be written as $f=\sum_{\alpha\in\mathbb{Z}^\mathbb{N}} f(\alpha)\delta_\alpha$ and so 
\begin{align*} 
\varphi(f) =\varphi\left(\sum_{\alpha\in\mathbb{Z}^\mathbb{N}} f(\alpha)\delta_\alpha \right) 
&= \sum_{\alpha\in\mathbb{Z}^\mathbb{N}} f(\alpha)\varphi(\delta_\alpha) \\
&= \sum_{\alpha\in\mathbb{Z}^\mathbb{N}} f(\alpha)\varphi(\delta_{e_1}^{\alpha_1} \star \delta_{e_2}^{\alpha_2} \star \cdots \star \delta_{e_{n_\alpha}}^{\alpha_{n_\alpha}}) \\ 
&= \sum_{\alpha\in\mathbb{Z}^\mathbb{N}} f(\alpha)\varphi(\delta_{e_1}^{\alpha_1}) \varphi(\delta_{e_2}^{\alpha_2}) \cdots \varphi(\delta_{e_{n_\alpha}}^{\alpha_{n_\alpha}}) \\ 
&= \sum_{\alpha\in\mathbb{Z}^\mathbb{N}} f(\alpha) z_1^{\alpha_1} z_2^{\alpha_2} \cdots z_{n_\alpha}^{\alpha_{n_\alpha}} \\ 
&=\sum_{\alpha\in\mathbb{Z}^\mathbb{N}} f(\alpha) z^\alpha. 
\end{align*}
Thus, each $\varphi\in\Delta(\ell^p_\omega(\mathbb{Z}^\mathbb{N}))$ is of the form $\varphi_z$ for some $z\in(\mathbb{C}^\times)^\infty$. 

Let $z\in\Gamma_\omega$. If $f\in \ell^p_\omega(\mathbb{Z}^\mathbb{N})$, then
\begin{align*}
    |\varphi_z(f)|
    \leq \sum_{\alpha\in\mathbb{Z}^\mathbb{N}} |f(\alpha)| |z^\alpha| 
    &= \sum_{\alpha\in\mathbb{Z}^\mathbb{N}} |f(\alpha)| |z_1|^{\alpha_1} |z_2|^{\alpha_2} \cdots |z_{n_\alpha}|^{\alpha_{n_\alpha}} \\
    &\leq \left( \sum_{\alpha\in\mathbb{Z}^\mathbb{N}} |f(\alpha)| \omega(\alpha) \right)^\frac{p}{p} \\
    &\leq \|f\|_{\ell^p_\omega(\mathbb{Z}^\mathbb{N})}^\frac{1}{p}.
\end{align*}
So, $\varphi_z$ is well-defined, and it is easy to verify that $\varphi_z$ is a nonzero multiplicative linear map.

Conversely, if $z\in(\mathbb{C}^\times)^\infty$ is such that $\varphi_z\in\Delta(\ell^p_\omega(\mathbb{Z}^\mathbb{N}))$, then for each $\alpha\in\mathbb{Z^N}$,  we have $|z^\alpha|=|\varphi_z(\delta_\alpha)|\leq\|\delta_\alpha\|_{\ell^p_\omega(\mathbb{Z}^\mathbb{N})}^\frac{1}{p} = \omega(\alpha)$, that is, $z\in\Gamma_\omega$.
\end{proof}

Now, we give an example of weight for which $\Delta(\ell^p(\mathbb{Z}^\mathbb{N}))\neq\prod_{k\in\mathbb{N}} \Gamma(\rho_{1,k},\rho_{2,k})$.

\begin{example}
Define a weight $\omega$ on $\mathbb{Z}^\mathbb{N}$ by 
\begin{align*}
    \omega(\alpha)=e^{|\alpha_1|} + e^{|\alpha_2|} + \dots + e^{|\alpha_{n_\alpha}|} \quad (\alpha\in\mathbb{Z}^\mathbb{N}).
\end{align*}
Then, for each $k\in\mathbb{N}$,  
\begin{align*}
    &\rho_{1,k} = \sup_{\alpha_k<0} \omega(\alpha_k)^\frac{1}{\alpha_k} =\sup_{\alpha_k<0} \left(k-1+e^{-\alpha_k} \right)^\frac{1}{\alpha_k} = \frac{1}{e} \\
    \text{and} \quad 
    &\rho_{2,k} = \inf_{\alpha_k>0} \omega(\alpha_k)^\frac{1}{\alpha_k} =\inf_{\alpha_k>0} \left(k-1+e^{\alpha_k} \right)^\frac{1}{\alpha_k} = e.
\end{align*}
Take $z=(e,e,e,\dots)\in\prod_{k\in\mathbb{N}} \Gamma(\rho_{1,k},\rho_{2,k})$. For $\alpha=(1,1,0,0,...)\in\mathbb{Z}^\mathbb{N}$, we have 
\begin{align*}
    |z_1|^{\alpha_1} |z_2|^{\alpha_2} = e^2 > 2e = e^{|\alpha_1|} + e^{|\alpha_2|} = \omega(\alpha)
\end{align*}
and so, $\varphi_z$ is not in $\Delta(\ell^p_\omega(\mathbb{Z}^\mathbb{N}))$.    
\end{example}

But we have the following result which states that the Gel'fand space is contained in the said product of annuli. 

\begin{proposition} \label{prop:gelfand space of lp(Z^N)}
Let $0<p\leq1$, and let $\omega$ be a weight on $\mathbb{Z}^\mathbb{N}$. Then 
$$\Delta(\ell^p(\mathbb{Z}^\mathbb{N})) = \Gamma_\omega \subset \prod_{k\in\mathbb{N}} \Gamma(\rho_{1,k},\rho_{2,k}).$$ 
Moreover, if $\omega(\alpha)=\prod_{i=1}^{n_\alpha}\omega_i(\alpha_i)$ $(\mathbf{0}\neq\alpha\in\mathbb{Z}^\mathbb{N})$ and $\omega(\mathbf0)=1$, where $\omega_i$ is a weight on $\mathbb{Z}$ for all $i\in\mathbb{N}$, then $\Gamma_\omega = \prod_{k\in\mathbb{N}} \Gamma(\rho_{1,k},\rho_{2,k}).$  
In particular, if $\omega$ satisfies the GRS-condition, then $\Gamma_\omega=\mathbb{T}^\infty= \prod_{k\in\mathbb{N}} \Gamma(\rho_{1,k},\rho_{2,k})$, where $\mathbb{T}^\infty=\{(z_1,z_2,z_3,\dots)\in\mathbb{C}^\infty:|z_i|=1\ \text{for all}\ i\in\mathbb{N}\}$.
\end{proposition}
\begin{proof}
Let $\varphi_z\in\Gamma_\omega$. Then $\prod_{j=1}^{n_\alpha} |z_j|^{\alpha_j}\leq\omega(\alpha)$. So, taking $\alpha=(0,0,\dots,0,\alpha_k,0,\dots)$, for each $k\in\mathbb{N}$, we get $|z_k|^{\alpha_k}\leq\omega(\alpha_k)$. So, $\omega(-\alpha_k)^\frac{1}{-\alpha_k}\leq|z_k|\leq\omega(\alpha_k)^\frac{1}{\alpha_k}$ for all $\alpha_k\in\mathbb{N}$. This implies that $z_k\in\Gamma(\rho_{1,k},\rho_{2,k})$ for each $k\in\mathbb{N}$. Thus, $\Gamma_\omega \subset \prod_{k\in\mathbb{N}} \Gamma(\rho_{1,k},\rho_{2,k})$.

Let $\omega$ be such that $\omega(\mathbf{0})=1$ and $\omega(\alpha)=\prod_{i=1}^{n_\alpha}\omega_i(\alpha_i)$ for nonzero $\alpha\in\mathbb{Z}^\mathbb{N}$, where $\omega_i$ is a weight on $\mathbb{Z}$ for all $i\in\mathbb{N}$, and let $z\in \prod_{k\in\mathbb{N}} \Gamma(\rho_{1,k},\rho_{2,k})$. Then, for each $k\in\mathbb{N}$, taking $\alpha=(0,0,\dots,0,\alpha_k,0,\dots)$, we have $\omega(-\alpha_k)^\frac{1}{-\alpha_k}\leq|z_k|\leq\omega(\alpha_k)^\frac{1}{\alpha_k}$, that is, 
\begin{align*}
    \omega_k(-\alpha_k)^\frac{-1}{\alpha_k} \prod_{j=1}^{k-1} \omega_j(0)^\frac{-1}{\alpha_k} \leq|z_k|\leq\omega_k(\alpha_k)^\frac{1}{\alpha_k} \prod_{j=1}^{k-1} \omega_j(0)^\frac{1}{\alpha_k}
\end{align*} 
for all $\alpha_k\in\mathbb{N}$. This gives $\displaystyle |z_k|^{\alpha_k} \leq \omega_k(\alpha_k) \prod_{\substack{i=1}}^{k-1} \omega_i(0)$ for all $\alpha_k\in\mathbb{Z}$ and $k\in\mathbb{N}$. Now, let $\alpha=(\alpha_1,\alpha_2,\dots,\alpha_n,0,0,\dots)$ with $n=n_\alpha$. Then, 
\begin{align*}
    \prod_{j=1}^n |z_j|^{\alpha_j} \leq \left(\prod_{j=1}^n \omega_j(\alpha_j)\right) \left( \prod_{\substack{i=1}}^{n} \omega_i(0)^{k-1} \right) = \left( \prod_{\substack{i=1}}^{n} \omega_i(0)^{k-1} \right) \omega(\alpha).
\end{align*} 
For each $m\in\mathbb{N}$, using submultiplicativity of weight, we get
\begin{align*}
    \left(\prod_{j=1}^n |z_j|^{\alpha_j}\right)^m = \prod_{j=1}^n |z_j|^{m\alpha_j} 
    \leq  \left( \prod_{\substack{i=1}}^{n} \omega_i(0)^{k-1} \right) \omega(m\alpha)
    \leq  \left( \prod_{\substack{i=1}}^{n} \omega_i(0)^{k-1} \right) \omega(\alpha)^m.
\end{align*} 
This implies that $\prod_{j=1}^n |z_j|^{\alpha_j} \leq  \left( \prod_{\substack{i=1}}^{n} \omega_i(0)^{k-1} \right)^\frac{1}{m} \omega(\alpha)$ for all $m\in\mathbb{N}$. Thus, $\prod_{j=1}^n |z_j|^{\alpha_j} \leq \omega(\alpha)$. Hence, $z\in\Gamma_\omega$. 

The particular case follows from the fact that if $\omega$ satisfies the GRS-condition, then $\rho_{1,k}=1=\rho_{2,k}$ for all $k\in\mathbb{N}$, and so $\prod_{k\in\mathbb{N}} \Gamma(\rho_{1,k},\rho_{2,k})=\mathbb{T}^\infty$.
\end{proof}

\subsection{Infinite variable version for BGS-type algebras}
Now, we move to the infinite variable analogue for inverse-closedness of BGS-type algebras. In fact, it is sufficient to state the vector-valued weighted analogue Wiener's theorem for infinite variables as rest of the analysis follows same line of arguments as given in Section \ref{sec:BGS}. For that, we first establish the requisite spaces.

Let $\mathcal{A}$ be a unital Banach algebra. For $0<p<\infty$ and a weight $\omega$ on $\mathbb{Z^N}$, define $\ell^p_\omega(\mathbb{Z^N},\mathcal{A})$ analogously. For $p>1$, further assume that $\omega$ satisfies
\begin{align} \label{eqn:algebra weight Z^N}
    \omega^{-p'}\star\omega^{-p'}\leq\omega^{-p'} \quad \text{and} \quad \sum_{\alpha\in\mathbb{Z^N}} \omega(\alpha)^{-p'}<\infty.
\end{align}
Then $\ell^p_\omega(\mathbb{Z}^\mathbb{N},\mathcal{A})$ is a Banach algebra. Let $\mathcal{A}_{p\omega}$ be the collection of all functions $F:\mathbb{T}^\infty\to\mathcal{A}$ such that it has the following representation
\begin{align*}
    F(z)=\sum_{\alpha\in\mathbb{Z^N}} f(n) z^\alpha \quad (z\in\mathbb{T}^\infty),
\end{align*}
where $f\in\ell^p_\omega(\mathbb{Z^N},\mathcal{A})$. Then we have the following theorem.

\begin{theorem} \cite[Theorem 4.1]{kbmul} \label{thm:thi, p<1}
Let $0<p\leq1$, $\mathcal{A}$ be a unital Banach algebra, $\omega$ be an admissible weight on $\mathbb{Z}^\mathbb{N}$, and let $F\in \mathcal{A}_{p\omega}$ be such that $F(z)$ is invertible in $\mathcal{A}$ for all $z\in\mathbb{T}^\infty$. Then $F$ is invertible in $\mathcal{A}_{p\omega}$.
\end{theorem}

Moreover, using the same techniques as in the above theorem and the particular case of the Proposition \ref{prop:gelfand space of lp(Z^N)}, we have the following theorem.

\begin{theorem} \label{thm:thi, p>1}
Let $1<p<\infty$, $\mathcal{A}$ be a unital Banach algebra, $\omega$ be an admissible weight on $\mathbb{Z}^\mathbb{N}$ satisfying \eqref{eqn:algebra weight Z^N}, and let $F\in \mathcal{A}_{p\omega}$ be such that $F(z)$ is invertible in $\mathcal{A}$ for all $z\in\mathbb{T}^\infty$. Then $F$ is invertible in $\mathcal{A}_{p\omega}$.
\end{theorem}

Let $\omega$ be a weight on $\mathbb{Z^N}$. For $0<p<\infty$, define the space $\mathcal{C}^p_\omega(\mathbb{Z^N})$ as the collection of all $B(\mathcal{H})$-valued matrices $A=[A_{k,l}]_{k,l\in\mathbb{Z^N}}$ that satisfies
\begin{align} 
    |A| = \sum_{\beta\in\mathbb{Z^N}} \left(\sup_{\alpha\in\mathbb{Z^N}} \|A_{\alpha,\alpha-\beta}\|\right)^p \omega(\beta)^p<\infty.
\end{align}
Then it is easy to verify that $\mathcal{C}^p_\omega(\mathbb{Z^N})$ is a $p$-Banach algebra for $0<p\leq1$. If $p>1$, then further assume $\omega$ satisfies \eqref{eqn:algebra weight Z^N}, and in that case $\mathcal{C}^p_\omega(\mathbb{Z^N})$ is a Banach algebra. We have the following theorem using above two theorems in place of Theorem \ref{thm:Wiener}.

\begin{theorem}
Let $0<p<\infty$, and let $\omega$ be a weight on $\mathbb{Z^N}$ such that for $p>1$, $\omega$ satisfies \eqref{eqn:algebra weight Z^N}. Then the algebra $\mathcal{C}^p_\omega(\mathbb{Z^N})$ is inverse-closed in $B(\ell^2(\mathbb{Z^N},\mathcal{H}))$ if and only if $\omega$ is admissible. In particular, $\mathcal{C}^p_\omega(\mathbb{Z^N})$ is symmetric if and only if $\omega$ is admissible.
\end{theorem}

\section*{Acknowledgment} The first author is grateful to the National Board for Higher Mathematics (NBHM), India, for the research grant (02011/39/2025/NBHM(R. P.)/R$\&$D II/16090). The second author gratefully acknowledges the Post Doctoral Fellowship under the ISIRD project 9--551/2023/IITRPR/10229 IIT Ropar. This work was partially supported by the FIST program of the Department of Science and Technology, Government of India, Reference No. SR/FST/MS--I/2018/22(C). No funding was received for conducting this study.






\bibliographystyle{amsplain}

\end{document}